\documentclass{amsart}

\usepackage{amsmath,amsfonts,amsthm,amssymb,amscd}
\usepackage{hyperref}

\usepackage{tikz}
\usepackage{tikz-cd}
\usetikzlibrary{positioning}
\usetikzlibrary{matrix,arrows,decorations.pathmorphing}

\usepackage[alphabetic]{amsrefs}

\theoremstyle{plain}
\newtheorem{theorem}{Theorem}[section]
\newtheorem{proposition}[theorem]{Proposition}
\newtheorem{step}{Step}
\newtheorem{lemma}[theorem]{Lemma}
\newtheorem{question}[theorem]{Question}
\newtheorem{corollary}[theorem]{Corollary}

\theoremstyle{definition}
 \newtheorem{definition}[theorem]{Definition}

\newtheorem{remark}[theorem]{Remark}

\newcommand{\R}{\mathbb{R}}  
\newcommand{\Q}{\mathbb{Q}}  
\newcommand{\Z}{\mathbb{Z}} 
\newcommand{\N}{\mathbb{N}}

\newcommand{\ZG}{\mathbb{Z}G}

\newcommand{\CAT}{\mathsf{CAT}}  

\DeclareMathOperator{\cd}{\mathsf{cd}}
\DeclareMathOperator{\gd}{\mathsf{gd}}

\DeclareMathOperator{\FV}{\mathsf{FV}}

\title[Subgroups of word hyperbolic groups in dimension $2$]{Subgroups of word hyperbolic groups in rational dimension $2$}

\author{Shivam Arora}
\author{Eduardo Mart\'inez-Pedroza}
\address{Department of Mathematics and Statistics. 
Memorial University of Newfoundland. St. John's, NL. Canada}
\email{sarora17@mun.ca}
\email{eduardo.martinez@mun.ca}
 
\date{\today}

\begin{document}

\begin{abstract} 
A result of Gersten states that  if $G$ is a hyperbolic group with integral cohomological dimension $\cd_{\Z}(G)=2$ then every finitely presented subgroup is hyperbolic. We generalize this result for the rational case $\cd_{\Q}(G)=2$. In particular, our result applies to the class of torsion-free hyperbolic groups $G$ with $\cd_{\Z}(G)=3$ and $\cd_{\Q}(G)=2$ discovered by Bestvina and Mess. 
\end{abstract}

\maketitle 

\section{Introduction}

The cohomological dimension $\cd_R(G)$ of a group $G$ with respect to a ring $R$ is less than or equal to $n$ if the trivial $RG$-module $R$ has a projective resolution of length $n$. Let $\Q$ denote the field of rational numbers. The main result of this note:

\begin{theorem}\label{thm:main}
Let $G$ be a hyperbolic group such that $\cd_\Q(G)\leq 2$. If $H$ is a finitely presented subgroup, then $H$ is hyperbolic. 
\end{theorem}

The analogous statement for $\cd_\Z(G)$ is a result of Steve Gersten that we recover as a consequence of the inequality
\[ \cd_{\Q}(G)\leq \cd_{\Z}(G).\]

\begin{corollary}[Gersten]\cite[Theorem 5.4]{Ge96} \label{cor:Gersten}
Let $G$ be a hyperbolic group such that $\cd_\Z(G)=2$. If $H$ is a   finitely presented subgroup, then $H$ is hyperbolic.
\end{corollary}

The first motivation to generalize Gersten's result to the rational case is the existence of hyperbolic groups of integral cohomological dimension three and rational cohomological dimension two. The nature of finitely presented subgroups of groups in this class was not known. The first examples of such groups  were discovered by Bestvina and Mess~\cite{BeMe91} based on methods by Davis and Januszkiewicz~\cite{DaJa91}. The class also  contains finite index subgroups of hyperbolic Coxeter groups, examples that were discovered by Dranishnikov~\cite[Corollary 2.3]{Dr99}. We recall the nature of Bestvina-Mess examples in the following corollary. 

\begin{corollary}\cite{BeMe91}
Let $X$ be a finite  polyhedral $3$-complex such that
\begin{itemize}
\item $X$ admits piecewise constant negative curvature cellular structure satisfying Gromov's link condition, and
\item $X$ is a $3$-manifold (without boundary) in the complement of a single vertex whose link is a non-orientable closed surface. 
\end{itemize}
If $G=\pi_1 X$ then $\cd_\Q(G)=2$, $\cd_\Z(G)=3$ and any finitely presented subgroup of $G$ is hyperbolic.
\end{corollary}

The statement of Corollary~\ref{cor:Gersten} is sharp in the sense that there exist  hyperbolic groups of integral cohomological dimension three containing finitely presented subgroups that are not hyperbolic, the first example was found by Noel Brady~\cite{Br99}. More recently, infinite families of hyperbolic groups of integral cohomological dimension three containing non-hyperbolic finitely presented subgroups have been constructed, see for example~\cite{Kr18}.

\begin{corollary}
If $G$ is a hyperbolic group such that $\cd_\Z(G)=3$ and it contains a non-hyperbolic  finitely presented subgroup, then $\cd_{\Q}(G) = \cd_{\Z}(G)$.
\end{corollary}

A second motivation of this project was to generalize Gersten's result to groups admitting torsion, specifically, to the class of hyperbolic groups $G$ admitting a $2$-dimensional classifying space for proper actions $\underline{E}G$. Recall that a model for $\underline{E}G$ is a $G$-CW-complex $X$  with the property that for each subgroup $H$ the subcomplex of fixed points is contractible if $H$ is finite, and empty if $H$ is infinite. The minimal dimension of a model for $\underline{E}G$ is denoted by $\underline{\gd}(G)$. Considering the cellular chain complex with rational coefficients of a model for $\underline{E}G$ with minimal dimension shows that
\[ \cd_\Q (G) \leq \underline{\gd}(G).\]
This inequality implies the following corollary. 

\begin{corollary}\label{cor:ProperActions}
If $G$ is a hyperbolic group such that $\underline{\gd}(G)\leq 2$, then any finitely presented subgroup is hyperbolic. 
\end{corollary}

The statement of Corollary~\ref{cor:ProperActions} was known in the following cases: 
\begin{itemize}
    \item If $G$ admits a $\CAT(-1)$ $2$-dimensional model for $\underline{E}G$, see~\cite[Corollary 1.5]{HaMa14}.  
    \item If $G$ admits a $2$-dimensional model for $\underline{E}G$, and $H$ is finitely presented with finitely many conjugacy classes of finite groups,  a consequence of~\cite[Theorem 1.3]{Ma17}.  
    \item If $G$ is a hyperbolic small cancellation group of type $C(7)$, $C(5){-}T(4)$, $C(4){-}T(5)$, $C(3){-}T(7)$ or $C'(1/6)$, see~\cite[Theorem 7.6]{Ge96}. 
\end{itemize}
We remark that for a group $G$ satisfying the hypothesis of Corollary~\ref{cor:ProperActions}, the conclusion follows from Gersten's result~\ref{cor:Gersten} if, in addition, $G$ is assumed to be virtually torsion free.  It is an outstanding question whether hyperbolic groups are virtually torsion free~\cite{KaWi00}.

\begin{remark}
During the refereing process of this manuscript, a generalization of Theorem~\ref{thm:main} was proved in context of totally disconnected locally compact hyperbolic groups~\cite{ACCM19}. 
\end{remark}

\subsection*{Homological filling functions and the Proof of Theorem~\ref{thm:main} }

Let $R$ be a subgring of $\Q$. The $(n+1)$-dimensional homological Filling Volume function over $R$ of a cellular complex $X$ is a function $\FV_{X, R}^{n+1} \colon \N \to \R$  describing the minimal volume required to fill integral cellular $n$-cycles with cellular $(n+1)$-chains with coefficients in $R$.
 
For a group $G$ with a $K(G,1)$ model $X$ with finite $(n+1)$-skeleton, the $(n+1)$-dimensional homological Filling Volume function over $R$ of $G$, denoted by $\FV_{G, R}^{n+1}$, is defined as $\FV_{\tilde X, R}^{n+1}$ where $\tilde X$ is the universal cover of $X$. This function depends only of the group $G$ up to the equivalence relation on the set of non-decreasing  functions $\N\to\R$ defined as $f\sim g$ if and only if $f\preceq g$ and $g\preceq f$, where $f\preceq g$ means there is $C>0$ such that for all $k\in\N$, 
 \[ f(k) \leq Cg(Ck+C)+Ck+C.\]

Recall that a group $G$ is of type $R\mbox{-}FP_n$ if the trivial $RG$-module $R$ admits a partial projective resolution 
 \[ P_n \to \cdots  \to P_2\to P_1\to P_0\to R\to 0\]
 where each $P_i$ is a finitely generated $RG$-module.  In~\cite{HM16}, it is shown that to define $\FV_{G, \Z}^{n+1}$ is enough to assume that the group $G$ is of type $\Z\mbox{-}FP_{n+1}$. We prove that the same statement holds for $\FV_{G, R}^{n+1}$ in Section~\ref{sec:defintion}. The main technical result of this note is the following.
 
\begin{theorem}\label{thm:main2}
Let $R$ be a subring of $\Q$. Let $G$ be a group of type $R \mbox{-}FP_{n+1}$ and suppose $\cd_R(G)=n+1$.  Let $H \leq  G$ be a subgroup of type $R\mbox{-}FP_{n+1}$. Then there is a constant $C>0$ such that for all $k$
\[ FV_{H, R}^{n+1} \preceq FV_{G, R}^{n+1}.\]
\end{theorem}
This theorem generalizes the main result of~\cite{HM16}, by considering an arbitrary subgring of the rational numbers  instead of only the ring of integers, and by replacing the topological assumptions $F_{n+1}$ on $G$ and $H$ with the weaker hypothesis $R\mbox{-}FP_{n+1}$. 

The main result of this note, Theorem~\ref{thm:main}, is a consequence of Theorem~\ref{thm:main2} and the characterization of hyperbolic groups stated below, which is credited to Gersten~\cite{Ge96a}. This characterization was revised by Mineyev~\cite[Theorem 7, statements (0) and (2)]{Mi02}, and it  was also revisited by Groves and Manning in~\cite[Theorem 2.30]{GrMa08}.  

\begin{theorem}\cite[Theorem 7]{Mi02} \cite[Theorem 2.30]{GrMa08}\label{thm:char}
A group $G$ is hyperbolic if and only if $G$ is finitely presented and the rational filling function $\FV_{G, \Q}^2$ is bounded by a linear function, i.e., $\FV_{G, \Q}^2(k) \preceq k$.
\end{theorem}

\begin{proof}[Proof of Theorem~\ref{thm:main}]
Let $G$ be a hyperbolic group such that $\cd_\Q(G)=2$, and let $H$ be  a finitely presented subgroup.  Theorem~\ref{thm:char} implies that $\FV_{G, \Q}^2$ is bounded by a linear function.  By Theorem~\ref{thm:main2},   $\FV_{H, \Q}^2 \preceq \FV_{G, \Q}^2$. It follows $\FV_{H, \Q}^2$ is bounded by a linear function. Then Theorem~\ref{thm:char} implies that $H$ is a hyperbolic group. 
\end{proof}

In view of Theorem~\ref{thm:char}, we raised the following question.
\begin{question}
	Let $G$ be a $\Q\mbox{-}FP_2$ group and suppose $FV_{G,\Q}^2$ is bounded by a linear function. Is $G$ a hyperbolic group?
\end{question}
The analogous question obtained by replacing $\Q$ with $\Z$ is known to have a positive answer~\cite[Theorem 5.2]{Ge96}. One motivation behind this question is that a positive answer would imply that in  Theorem\ref{thm:main}, $H$ can be assumed to be $\Q\mbox{-}FP_2$ instead of being finitely presented. Recall that $\Q\mbox{-}FP_2$ condition is weaker than being finitely presented, see the examples in~\cite{BeBr97}.

 The rest of the note is devoted to the definition of homological filling function  and the proof of Theorem~\ref{thm:main2}. The argument is relatively self-contained, and uses and simplifies ideas from~\cite{HM16}. The main contributions of the article beside the results stated above are:
\begin{enumerate}
\item The definition of filling functions for arbitrary subdomains of the rationals, since the definition in~\cite{HM16} does not generalize directly, and 

\item  the replacement of topological arguments in~\cite{HM16} by algebraic ones that allow us to prove certain statements under the weaker homological finiteness condition $R\mbox{-}FP_{n+1}$ instead of the topological assumption $F_{n+1}$; see Proposition~\ref{prop:MappingCylinder} which is a construction based on the homological mapping cylinders,  and  Remark~\ref{rem:FH}.   
\end{enumerate}

\subsection*{Organization} Preliminary definitions are included in Section~\ref{sec:preliminaries}, specifically the notions of filling norms and bounded morphisms on modules over arbitrary normed rings.  Section~\ref{sec:defintion}  discusses the generaliation of  homological filling functions defined over  arbitrary subdomains of the rational numbers. The last section contains the proof of  Theorem~\ref{thm:main2}. 

\subsection*{Acknowledgments.} The authors thank Mladen Bestvina, Ilaria Castellano, and the referee for 
comments and corrections. The second author acknowledges funding by the Natural Sciences and 
Engineering Research Council of Canada, NSERC.

\section{Filling Norms, Bounded morphisms} \label{sec:preliminaries}

All rings considered in this article have a multiplicative identity. Let $R$ be a ring and let $\R$ denote the ordered field of real numbers. A norm on $R$ is a function $| \cdot | \colon R \rightarrow \mathbb R$ such that for any $r, r'\in R$ 
\begin{itemize}
	\item $| r | \geq 0$ with equality if and only if $r = 0$,  
	\item $ | r + r' | \leq | r | + | r' | $, and
	\item $| r_1  r_2 | \leq |r_1| | r_2 |$ for $r_1,r_2  \in R$.
\end{itemize}
A \emph{normed ring} is a ring equipped with a norm.

From here on, assume that $R$ is a normed ring. A norm on an $R$-module $M$ is a function $\| \cdot \| \colon M \rightarrow \mathbb R$ such that for any $m,m'\in M$ and $r\in R$
\begin{itemize}
	\item $\| m \| \geq 0$ with equality if and only if $m = 0$, 
	\item $ \| m  + m' \| \leq \| m \| + \| m' \| $, and
	\item $\| r  m \| \leq |r|  \| m \|$.
\end{itemize}
A function $M\to \R$ that satisfies the last two conditions and has only non-negative values is called a pseudo-norm.

The \emph{$\ell_1$-norm} on a free $R$-module $F$ with fixed basis $\Lambda$ is defined as 
\[\left \| \ \displaystyle \sum_{x \in \Lambda} r_x x  \ \right \|_1 = \displaystyle \sum_{x \in \Lambda} | r_x |.\] 
A free $R$-module with fixed basis is called a \emph{based free module}.

\begin{definition}[Filling norm]
	A \emph{filling norm} on a finitely generated $R$-module $M$ is  defined as follows. Let $\rho \colon F \to M$ be a surjective morphism of $R$-modules where $F$ is a finitely generated free $R$-module with fixed basis $\Lambda$ and induced $\ell_1$-norm $\|\cdot\|_1$.  The \emph{filling norm on $M$} induced by $\rho$ and $\Lambda$ is defined as  
	\[ \|m\|_M= \inf \{\|x\|_1 \colon x \in F, \rho(x)=m \}.\]    
\end{definition}

\begin{remark} \label{rem:regularity}
	The following statements can be easily verified.
	\begin{enumerate}
		\item An $\ell_1$-norm $\|\cdot \|_1$ on a finitely generated free $R$-module $F$ is a filling norm.
		
		\item A filling norm $\|\cdot \|$ on a finitely generated $R$-module $M$ is a pseudo-norm, and is regular in the sense that
		\[ \|r m\| = |r|  \|m\| \]
		for  any $m\in M$ and $r\in R$ such that $r$ is a unit and $|r| |r^{-1}|=1$.
	\end{enumerate}
\end{remark}

\begin{definition}[Bounded Morphism]
	A morphism $f\colon M \rightarrow N$ between $R$-modules with norms $\|\cdot\|_M$ and $\|\cdot\|_N$ respectively is called \emph{bounded (with respect to these norms)} if there exists a fixed constant $C > 0$ such that $\|f(a)\|_N \leq C\|a\|_M $ for all $a \in M$.
\end{definition} 

The following lemma appears in~\cite{Ma17} for the case that $R$ is a group ring. 
The proof for an arbitrary ring is analogous, we have included the argument for the convenience of the reader.

\begin{lemma}\label{lem:ProjBounded} \cite[Lemma 4.6]{Ma17}
	Morphisms between finitely generated $R$-modules are bounded with respect to filling norms.
\end{lemma}
\begin{proof}
	First observe that if $\tilde\varphi \colon A \to B$ is a morphism between finitely generated based free $R$-modules, then for $a\in A$,  \[ \| \tilde \varphi(a) \|_B \leq C  \| a \|_A,\] where $\|\cdot\|_A$ and $\|\cdot\|_B$  are the corresponding $\ell_1$-norms, the constant $C$ is defined as $\max\{\| \tilde \varphi (a) \|_B \colon  a\in \Lambda \}$ where $\Lambda$ is the fixed basis of $A$.
	
	Now we prove the statement of the lemma. Let $\varphi \colon P \rightarrow Q$ be a morphism between finitely generated $R$-modules, and let $\| \cdot \|_P$ and $\| \cdot \|_Q$ denote filling norms on $P$ and $Q$ respectively.  Suppose $A$ is a finitely generated based free $R$-module and that  $\rho\colon A\to P$ induces the filling norm $\|\cdot\|_P$, and analogously assume that $\rho'\colon B\to Q$ induces the filling norm $\|\cdot\|_Q$. Then, since $A$ is free, there is a morphism $\tilde \varphi\colon A \to B$ such that $\varphi \circ \rho = \rho' \circ \tilde \varphi$. Let $C$ be the constant  for $\tilde \varphi$ defined above. Let $p\in P$ and note that for any $a\in A$ such that $\rho(a)=p$, 
	\[ \| \varphi (p) \|_Q \leq \| \tilde \varphi (a) \|_B \leq C \| a\|_A .\]
Hence $\| \varphi (p) \|_Q \leq C \|p\|_P$. 
\end{proof}

Two norms $\| \cdot \|$ and $\| \cdot \| '$ on an $R$-module $M$ are said to be \emph{equivalent} if there exists a  constant $C > 0$ such that for all $m \in M$ \[C^{-1} \| m \| \leq \| m \| ' \leq C \| m \|.\] By considering the identity function on a finitely generated module $M$, the previous lemma implies:

\begin{corollary}\label{cor:fillingnorms}
	Any two filling norms on a finitely generated $R$-module $M$ are equivalent.
\end{corollary}

\begin{remark}\label{prl} \label{inl}
	Let $M$ be a free $R$-module with basis $\Lambda$, and let $N$ be a free $R$-submodule generated by a finite subset $\Lambda' \subseteq \Lambda$. Consider the induced $\ell_1$-norms $\|\cdot\|_\Lambda$ and $\|\cdot\|_{\Lambda'}$ on $M$ and $N$ respectively.
	\begin{enumerate}
		\item The projection map $\pi \colon M \rightarrow N$ is bounded with respect to the induced $\ell_1$-norms.
		\item The inclusion map $\imath \colon N \rightarrow M$ preserves the induced $\ell_1$-norms, in particular, it is bounded.
	\end{enumerate}
\end{remark}

\begin{lemma}\label{lem:inclusion}
Let $N$ be a finitely generated module with filling norm $\|\cdot\|_N$. Suppose that $N$ is an internal direct  summand of a free module $F$ with an $\ell_1$-norm $\|\cdot\|_1$. Then $\|\cdot\|_N \sim \|\cdot\|_1$ on $N$.
\end{lemma}
\begin{proof}
 Since $N$ is a finitely generated module contained in $F$, there exist a finitely generated free submodule $I$ of $F$ which is an internal summand, $F= I \oplus J$,  such that $N \subseteq I$, and the restriction of $\|\cdot\|_1$ to $I$ is an $\ell_1$-norm on $I$. Let $\iota \colon N \to I $ denote the inclusion and $\phi \colon F \to N$ denote the projection. By Lemma~\ref{lem:ProjBounded}, both $\phi|_{I} \colon I \to N$  and $\iota \colon N \to I $ are bounded morphisms with respect to the norms $\|\cdot\|_1$ and $\|\cdot\|_N$; let $C_1$ and $C_2$ be the corresponding constants. Then
\[ \|n\|_N = \|\phi(\iota(n))\|_N \leq C_1\|\iota(n)\|_1 \leq C_2C_1\|n\|_N\]
for all $n\in N$, and hence $\|\cdot\|_N \sim \|\cdot\|_1$ on $N$.
\end{proof}

For the rest of this section, let $G$ be a group, let $H$ be a subgroup, and as above let $R$ be a ring with norm $|\cdot|$.

\begin{remark}\label{ghn}
Let $M$ be a free $RG$-module with $\ell_1$-norm $\|\cdot \|_{\Lambda}$ induced by a free basis set $\Lambda$. Then $M$ is a free $RH$-module and there exist a free $RH$-basis $\Lambda_H$ of $M$ such that the induced $\ell_1$-norms  $\|\cdot\|_{\Lambda}$ and  $\|\cdot\|_{\Lambda_H}$ are equal.
	
Indeed, if $S$ is a right transversal of the subgroup $H$ in $G$, then $\Lambda_H=\{ gx \colon x\in \Lambda, g\in S\}$ is a free $RH$-basis  of $M$ as an $H$-module, and the statement about the $\ell_1$-norms holds.
\end{remark}

\begin{lemma}\label{prop:RetractionLemma}
Let $M$ be a finitely generated and projective $RG$-module with filling norm $\|\cdot\|_M$ and let $N$ be a finitely generated $RH$-module with filling norm $\|\cdot\|_N$. Suppose that $N$ is a internal direct summand of $M$ as an $RH$-module. Then $\|\cdot\|_N \sim \|\cdot\|_M$ on $N$.
\end{lemma}
\begin{proof}
Let $F$ be a finitely generated free based module with $\ell_1$-norm $\|\cdot\|_1$, and let $\phi: F \to M$ be a surjective $RG$-morphism inducing filling norm $\|\cdot \|_M$. Since $M$ is projective, there exist an $RG$-morphism $j \colon M \to F$ such that $j \circ \phi = \textup{id}_M$.  Lemma~\ref{lem:ProjBounded} implies that $j$ and $\phi$ are bounded $RG$-morphisms. Therefore $\|\cdot\|_M \sim \|\cdot\|_1$.  Now consider $F$ as an $RH$-module with same $\ell_1$-norm $\|\cdot\|_1$, see Remark~\ref{ghn}. Since $N$ is a direct summand of $M$ as an $RH$-module, it is a direct summand of $F$ as an $RH$-module. Then Lemma~\ref{lem:inclusion} implies $\|\cdot\|_N \sim \|\cdot\|_1$ on $N$.
\end{proof}

\section{Definition of Homological Filling Functions} \label{sec:defintion}

In this section $R$ denotes  a subring of the rational numbers with the absolute value as a norm.  Let $G$ be a group. The group ring $RG$ is a free abelian module over $R$, observe that $RG$ is a normed ring with $\ell_1$-norm induced by the free $R$-basis $G$. From now on, we consider $RG$ as a normed ring with this norm.
\begin{definition}[Integral part]
	Let $P$ be a finitely generated $RG$-module. An integral part of $P$ is a $\ZG$-submodule $A$ which is finitely generated as a  $\ZG$-module, and $A$ generates $P$ as an  $RG$-module.
\end{definition}

From here on, $[0, \infty]$ denotes the set of non-negative real numbers and infinity. The order relation as well as the addition operations are extended in the natural way. 

\begin{definition} \label{def:algFVG} The  \emph{$n^\text{th}$--filling function} of a group $G$ of type $R\mbox{-}FP_{n+1}$, 
 \[FV_{G,R}^{n+1}\colon \N \to [0, \infty],\] 
is defined as follows.  	Let 
\begin{equation} \label{eq:resolution}
	P_{n+1} \overset{\partial_{n+1}} \longrightarrow P_n \overset{\partial_{n}} \longrightarrow \dots \overset{\partial_{2}} \longrightarrow P_1 \overset{\partial_{1}} \longrightarrow P_0 \longrightarrow R \rightarrow 0,
	\end{equation}
	be a partial projective resolution of finitely generated $RG$-modules of the trivial $RG$-module $R$.  Let $K_n$ be an integral part for $\ker(\partial_n)$,  let $\|\cdot \|_{P_n}$ and $\|\cdot \|_{P_{n+1}}$ be filling norms for $P_n$ and $P_{n+1}$ respectively. 
	Then
	\[FV_{G,R}^{n+1}(k) = \sup \left \{ \ \| \gamma \|_{\partial_{n+1}} \colon \gamma \in K_n, \ \| \gamma \|_{P_n} \leq k \ \right \},\]
	where 
	\[ \| \gamma \|_{\partial_{n+1}} = \inf \left\{  \| \mu\|_{P_{n+1}} \colon \mu \in P_{n+1},\  \partial_{n+1} (\mu) =\gamma \right\} .\]
By convention, define the supremum of the empty set as zero.
\end{definition}

See Remark~\ref{rem:finiteness} on finitiness of $FV_{G, R}^{n+1}$.  The rest of this section discusses the proof of the following theorem, which generalizes~\cite[Theorem 3.5]{HM16}. Consider the equivalence relation on the set of non-decreasing  functions $\N\to[0, \infty]$ defined as $f\sim g$ if and only if $f\preceq g$ and $g\preceq f$, where $f\preceq g$ means there is $C>0$ such that for all $k\in\N$, 
 \[ f(k) \leq Cg(Ck+C)+Ck+C.\]

\begin{theorem}\label{AlgebraicDef}  Let $G$ be a group of type $R\mbox{-}FP_{n+1}$. Then the $n^\text{th}$--filling function $FV_{G,R}^{n+1}$ of $G$ is well defined up to the equivalence relation $\sim$. 
\end{theorem}

The proof of Theorem~\ref{AlgebraicDef} relies on the  following  basic structure theorem for subgrings of $\Q$.  

\begin{proposition}\label{prop:R-structure}
	Let $R$ be a subring of $\Q$. Then there is a set $S$ of prime numbers in $\Z_+$ such that $R$  consists of all fractions $\frac{a}{b}$ where $a\in \Z$ and $b$ is product of powers of elements of $S$.
\end{proposition}

In the following proposition, which is a consequence of Proposition~\ref{prop:R-structure}, we use the convention that for an element $a$ of an $RG$-module $A$, and any $r\in R$,  $ra$ denotes the element $(re)a \in A$ where  $e$ is the identity element of $G$; moreover, the ring of integers $\Z$ is naturally identified with the subring of $RG$ via  $m\mapsto me$. 

\begin{proposition} \label{prop:integral}
	Let $P$ and $Q$ be finitely generated $RG$-modules. Then
	\begin{enumerate}
		\item If $A$ is an integral part, then for all units $r\in R$, $rA= \{ra \colon a \in A\}$ is an integral part.
		\item If $f\colon P \to Q$ is a morphism of $RG$-modules, and $A$ and $B$ are integral parts of $P$ and $Q$ respectively, then there exists a positive integer $m$ which is a unit in $RG$ and such that $ f(mA) \subseteq B$.
	\end{enumerate}
\end{proposition}
\begin{proof}
	The first statement is immediate from the definition.  	For the second statement.
	Let $S$ be a finite generating set of $A$ as a $\ZG$-module, and observe that $S$ generates $P$ as an $RG$-module. Let $F(S)$ be the free $RG$-module on $S$, let $\phi\colon F(S) \to P$, and let $C$ be the $\ZG$-submodule of $F(S)$ generated by $S$, and observe that $\phi(C)=A$. Analogolusly, let $T$ be a finite generating set of $B$ as a $\ZG$-module, let $\psi\colon F(T) \to Q$, and let $C'$ be the $\ZG$-submodule of $F(T)$ generated by $T$, and note that $\psi(C')=B$.
	
	Since $F(S)$ is free, there is an $RG$-morphism $\eta \colon F(S) \to F(T)$ such that the following diagram commutes.
	\begin{equation}
	\begin{tikzpicture}[>=angle 90]
	\matrix(a)[matrix of math nodes,
	row sep=2.5em, column sep=2.5em,
	text height=1.5ex, text depth=0.25ex]
	{F(S)&F(T)\\
		P&Q\\};
	\path[->](a-1-1) edge node[above]{$\eta$} (a-1-2);
	\path[ ->>](a-1-1) edge node[left]{$\phi$}(a-2-1);
	\path[->>](a-1-2) edge node[right]{$\psi$}(a-2-2);
	\path[ ->](a-2-1) edge node[below]{$f$} (a-2-2);
	\end{tikzpicture}
	\end{equation}
	
	Note that $\eta \colon F(S) \to F(T)$ is described by a finite matrix with entries in $RG$. By Proposition~\ref{prop:R-structure}, there is an integer $m$, which is a unit in $R$, such that the morphism $m.\eta \colon F(S) \to F(T)$ given by $\alpha \mapsto m\alpha$ has the property that $\eta(C) \subseteq C'$. By commutativity of the diagram $f \circ (m\phi) = \psi \circ (m\eta)$ and therefore $f(mA)\subseteq B$.
\end{proof}

The following lemma is a strengthening of  Proposition~\ref{prop:integral} that will be used in the last section.
\begin{lemma} \label{lem: integral part H}
	Let $H \leqslant G$ be a subgroup and let $P$ and $Q$ be finitely generated $RH$ and $RG$ modules respectively. If $f\colon P \to Q$ is an $RH$-morphism, and $A$ and $B$ are integral parts of $P$ and $Q$ respectively, then there exists a positive integer $m$, which is a unit in $R$, such that $ f(mA) \subseteq B$.
\end{lemma}
\begin{proof}
	Considering $Q$ as an $RH$-module, the proof proceeds similar to \ref{prop:integral} except that here $F(T)$ is infinitely generated and so the matrix is infinite. But observe that only finitely many entries are non-zero, so the same argument holds.
\end{proof}

\begin{proof}[Proof of Theorem~\ref{AlgebraicDef}]
	The proof is divided into two steps. The second step is a small variation of the argument in~\cite[Proof of Theorem 3.5]{HM16} for which we only remark the changes. 
	
	\begin{step}
		$FV_G^{n+1}$ (up to equivalence) does not depend on the choice of the integral part $K_n$.
	\end{step} Let $A$ and $B$ be two integral parts of $K_n$, and let $FV_A$ and $FV_B$ denote the corresponding $n^{th}$-filling functions of $G$. By Proposition \ref{prop:integral}, there exists a positive integer $m$, that is a unit in $RG$, such that $m.A \subseteq B$. Let $\gamma \in A$ such that  $\|\gamma\|_{P_n} \leq k$. 
	Then, since $m$ is a unit and $|m| |m^{-1}|=1$, $\|\gamma\|_{\delta_{n+1}} = \frac{1}{m}\|m\gamma\|_{\delta_{n+1}}$ and $\|m\gamma\|_{P_n} = m\|\gamma\|_{P_n} \leq mk $;  see  Remark~\ref{rem:regularity}. Observe that $m\gamma \in B$ therefore $\|\gamma\|_{\delta_{n+1}} \leq \frac{1}{m}FV_B(mk)$. Since $\gamma$ was arbitrary, $FV_A(k) \leq \frac{1}{m}FV_B(mk)$. By symmetry we get the other inequality. 	                                                                                                                    
	\begin{step}$FV_G^{n+1}$ (up to equivalence) does not depend on the choice of the resolution~\eqref{eq:resolution}.
	\end{step}
	Let $(P_\ast, \partial_\ast)$ and $(Q_\ast, \delta_\ast)$ be a pair of resolutions as in~\eqref{eq:resolution}. Since any two projective resolutions of $R$ are chain homotopy equivalent, there exist chain maps $f_i \colon P_i \rightarrow Q_i, \ g_i \colon Q_i \rightarrow P_i$, and a map $h_i \colon P_i \rightarrow P_{i+1}$ such that \[ \partial_{i+1} \circ h_i + h_{i-1} \circ \partial_i = g_i \circ f_i - Id.\]
	By Proposition \ref{prop:integral}, there exist integral parts $K_n$ and $K'_n$ of $ker(\partial_n)$ and $ker(\delta_n)$ respectively, such that $f_n(K_n) \subseteq K'_n$. This ensures that the same argument in~\cite[Proof of Theorem 3.5]{HM16} works except for a minor change in the choice of the element named $\beta$ in the cited proof. Replace it by the following: “For $\epsilon >0$, choose $\beta \in Q_{n+1}$ such that $\delta_{n+1}(\beta) = f_n (\alpha )$ and  $\| \beta \|_{Q_{n+1}} < \|f_n(\alpha) \|_{\delta_{n+1}} + \epsilon ”$. The rest of the proof proceeds in the same manner.
\end{proof}

\begin{remark}[Topological interpretation of filling functions]
Assume $G$ admits a $K(G, 1)$ model $X$ with finite $(n+1)$-skeleton. The augmented cellular chain complex $C_*(\tilde X, R)$ of the universal cover $\widetilde X$ of $X$ is a projective resolution of the trivial $RG$-module $R$ by free modules. By considering the $\ell_1$-norm of 
	$C_i(\tilde X,R)$ induced by the basis consisting of $i$-dimensional cells of $\widetilde X$, the definition of $\FV_{G,R}^{n+1}$ using this resolution provides the interpretation $\FV_{G,R}^{n+1}$ as the minimal volume required to fill integral $n$-cycles with  $(n+1)$-cellular chains with coefficients in $R$. Observe that 
\begin{equation}\label{eq:finiteness}  FV^{n+1}_{G,R} \leq FV^{n+1}_{G,\Z} \end{equation}
\end{remark}

\begin{remark}\label{rem:finiteness}[Finiteness of $FV^{n+1}_{G, R}$]
Assume that $G$ admits a $K(G, 1)$ model $X$ with finite $(n+1)$-skeleton. By the main result of~\cite{FlMa18}, for every positive integer $k$, $FV_{G,\Z}^{n+1}(k) <\infty$. Then equation~\eqref{eq:finiteness} implies that $FV_{G,R}^{n+1}(k) <\infty$ for any $k\geq 0$. 
\end{remark}

A positive answer to the following question in the case that $R=\Z$ is given in~\cite{FlMa18}. 
\begin{question}
Suppose that $G$ is of type $R\mbox{-}FP_{n+1}$. Is $FV_{G,R}^{n+1}(k)<\infty$  for all $k\in \N$?
\end{question}

\begin{remark}[On the use integral part in Definition~\ref{def:algFVG}]
	We note that the filling function $\FV_{G,\Z}^{n+1}$ was defined in~\cite{HM16} by considering $ker(\partial_n)$ in lieu of its integral part. This approach does not generalize in order to define  $\FV_{G,\Q}^{n+1}$   as the following example illustrates. Consider the group presentation $G = \langle x,y| [x,y]\rangle$ and let  $X$ be the universal cover of the presentation complex, i.e., the Cayley complex.  In $X$  consider the following cycles with rational coefficients $a_n= \frac{1}{4n}[x^ny^n]$ for $n \in \N$. Then $\|a_n\|_1= 1$ and by regularity $\|a_n\|_{\partial}=\frac{1}{4}n$, in particular 
\[ \max\{\|\gamma\|_{\partial_2}  \colon \gamma \in Z_n(\widetilde X, \Q), \| \gamma \|_1 \leq 1\} = \infty,\]
and hence the approach in~\cite{HM16} does not yield a well defined $FV^2_{G,\Q}(k)$. In contrast, using Definition~\ref{def:algFVG}, $FV^2_{G,\Q} \preceq FV^2_{G, \Z} \sim k^2$.
\end{remark}

\section{Proof of Theorem~\ref{thm:main2}} \label{sec:MappingCylinder}

The proof of Theorem~\ref{thm:main2} is discussed after the proof of the following proposition. 

\setcounter{step}{0}

\begin{proposition}  \label{prop:MappingCylinder}
	Suppose that $cd_R(G)=n+1$, $G$ is of type $R\mbox{-}FP_{n+1}$, and $H$ is a subgroup of $G$ of type $R\mbox{-}FP_{n+1}$.  Then for any  partial projective resolution of the trivial $RH$-module $R$ of finite type
	\begin{equation} \label{eq:02}
	Q_{n+1} \rightarrow Q_n \rightarrow \cdots \rightarrow Q_0 \rightarrow R \rightarrow 0,
	\end{equation}
	there is a projective resolution of the trivial $RG$-module $R$ of finite type
	\begin{equation} \label{eq:01}   
	0 \rightarrow M_{n+1}\rightarrow M_n \rightarrow \cdots \rightarrow M_0 \rightarrow R\rightarrow 0,
	\end{equation}
	an injective morphisms $\imath_i\colon Q_i\to M_i$ of $RH$-modules, $0\leq i\leq n$, such that 
	\begin{equation}\label{eq:03}
	\begin{tikzcd}[cells={nodes={minimum height=2em}}]
	Q_n \arrow[r]\arrow[d,"\imath_n"] & \cdots \arrow[r] & Q_1 \arrow[d, "\imath_1"] \arrow[r] &  Q_0 \arrow[r] \arrow[d, "\imath_0"] & R \arrow[d, "Id"] \\
	M_n \arrow[r] & \cdots \arrow[r] & M_1 \arrow[r] & M_0 \arrow[r] & R .
	\end{tikzcd}
	\end{equation} 
	is a commutative diagram of $RH$-modules, and the short exact sequences of $RH$-modules
	\begin{equation}\label{eq:07}
	\begin{tikzcd}[cells={nodes={minimum height=2em}}]
	0 \arrow[r] & Q_i \arrow[r, "\imath_i"] & M_i \arrow[r] & S_i \arrow[r] & 0
	\end{tikzcd}
	\end{equation}
	split. In particular each $S_i$ is a projective $RH$-module.
\end{proposition}

\begin{remark}\label{rem:FH}
	Proposition~\ref{prop:MappingCylinder} replaces topological arguments in~\cite{HM16}, based on work of Gersten~\cite{Ge96}, that use topological mapping cylinders. The arguments there are relatively less involved. In the generality that we are working, it is not possible to rely on this type of topological constructions. We would need free cocompact actions on  $(n+1)$-acyclic complexes for $G$ and $H$, they are not known to exist under our hypothesis.  Specifically, recall that a group $G$ is of type $FH_n$, if $G$ admits a cocompact action on an $n$-acyclic space $X$; in this case the action of $G$ on the cellular chain complex of $X$ induces a resolution of $\Z$ as a $\Z G$-module. Hence $FH_n$ implies $FP_n$. It is an open question whether groups of type $FP_n$ are of type $FH_n$ for $n\geq 3$, see~\cite{BeBr97}.
\end{remark}

The proof of the Proposition~\ref{prop:MappingCylinder} is an application of the \emph{mapping cylinder} of chain complexes from basic homological algebra that we recall below. 

Let $B_*=\{B_i,d_i\} $ and $C_*= \{C_i,d_i'\}$ be two chain complexes of modules over some fixed ring, and let $f\colon B_* \rightarrow C_*$ be a chain map. Then the mapping cylinder $M_*=\{M_i,d_i''\}$ is a chain complex where $M_i= C_i \oplus B_i \oplus B_{i-1} $ with 
\[ d_i''= \begin{pmatrix}
d'_i & 0 & -f_i \\
0 & d_i& Id_B\\
0 &0 & -d_i
\end{pmatrix}\]
Observe that, if both $B_*$ and $C_*$ consists only of finitely generated projective modules then the the same holds for $M_*$. The natural inclusion $C_* \hookrightarrow M_*$ given by $c\mapsto (c,0,0)$ is a chain homotopy equivalence. The chain homotopy inverse map $\kappa_*\colon M_* \to C_*$ is given by $(c,b,b') \mapsto c + f(b)$.  Let $\jmath_*\colon B_* \to M_*$ be the inclusion given by $b \mapsto (0,b,0)$.  It is an observation that the triangle
\begin{equation}\label{eq:Ilaria} 
\begin{tikzcd}[cells={nodes={minimum height=2em}}]
B_*\arrow[rr, "f_*"] \arrow[rd, "\jmath_*"] & &C_*\\  &M_*\arrow[ru, "\kappa_*"]& 
\end{tikzcd}
\end{equation}
commutes.  For background on mapping cylinders see~\cite{We94}.

\begin{proof}[Proof of Proposition~\ref{prop:MappingCylinder}] We split the proof into four steps.
	\begin{step}Definition of the resolution~\eqref{eq:01} as a mapping cylinder\end{step}
	
	Since $cd_R(G)=n+1$ and $G$ is of type $R\mbox{-}FP_{n+1}$, there is a projective resolution of $RG$-modules of finite type
	\begin{equation} \label{eq1}
	0\rightarrow P_{n+1} \rightarrow P_n \rightarrow .....\rightarrow
	P_0  \rightarrow R \rightarrow 0,
	\end{equation}
	see~\cite[pg.199, Prop. 6.1]{Br94}. 
	
	The group ring $RG$ is a free right $RH$-module. It follows that the extension of scalars functor from left $RH$-modules to left $RG$-modules $A \mapsto RG\otimes_{RH} A$ is exact. This functor also preserves finite generation  and projectiveness. From the given resolution~\eqref{eq:02}, we obtain  a partial projective resolution of the $RG$-module  $RG\otimes_{RH} R$ of finite type
	\begin{equation}  \label{eq3}
	RG\otimes_{RH} Q_n \rightarrow \cdots \rightarrow RG\otimes_{RH} Q_0 \rightarrow RG \otimes_{RH} R \rightarrow 0.
	\end{equation}
	
	Consider the $RG$-morphism $\phi \colon RG\otimes_{RH} R \to R$ induced by \begin{equation}\label{eq12} \phi\colon RG \times R \to R, \qquad  (s,r) \mapsto \epsilon(s)r, \end{equation} 
	where $\epsilon\colon RG \to R$ is the augmentation map,  $ \epsilon (\sum r_ig_i ) =  \sum r_i$.  Since each of the $RG$-modules $RG\otimes_{RH} Q_i$ is projective, there are $RG$-morphisms $f_i\colon RG \otimes_{RH} Q_i  \to  P_i$ such that 
	\begin{equation}\label{eq15}
	\begin{tikzcd}[cells={nodes={minimum height=2em}}]
	RG \otimes_{RH} Q_n \arrow[r]\arrow[d,"f_n"] & \cdots \arrow[r] & RG \otimes_{RH} Q_0 \arrow[r] \arrow[d, "f_0"] &   RG \otimes_{RH} R  \arrow[d, "\phi"] \\
	P_n \arrow[r] & \cdots \arrow[r] & P_0 \arrow[r] & R.
	\end{tikzcd}
	\end{equation} 
	is a commutative diagram, see~\cite[pg.22, Lemma 7.4]{Br94}.
	
Let $M_*=(M_i)$ be the mapping cylinder of the chain map $f=(f_i)$ where $f_i$ is the $RG$-morphism  defined above for $0\leq i\leq n$, $f_{n+1}$ is the morphism $0\to P_{n+1}$, and $f_i$ is the morphism $0\to 0$ for any other value of $i$.  
	
	Observe that  \[ M_i = P_i \oplus (RG \otimes_{RH} Q_i) \oplus (RG \otimes_{RH} Q_{i-1})\] for $1\leq i\leq n$, 
	$M_0=P_0 \oplus (RG \otimes_{RH} Q_0) \oplus 0 $, 
	$M_{n+1} = P_{n+1} \oplus 0 \oplus (RG \otimes_{RH} Q_n)$, and $M_i=0$ for any other value of $i$. Hence all $M_i$ are finitely generated and projective. 
	
	Let $P_*=(P_i)$ be the chain complex induced by~\eqref{eq1}, where $P_i=0$ for $i>n+1$ and $i<0$. Observe that $P_*$ is the target of the chain map $f$.  Since $P_*$ and $M_*$ are chain homotopic, 
	\begin{equation} \nonumber  
	0 \rightarrow M_{n+1}\rightarrow M_n \rightarrow \cdots \rightarrow M_0 \rightarrow R\rightarrow 0,
	\end{equation}
	is a projective resolution of finite type of the trivial $RG$-module $R$.

	\begin{step}
		Definition of the injective $RH$-morphisms $\imath_i\colon Q_i \to M_i$.
	\end{step}
	
	We have the following commutative diagram of $RH$-modules
	\begin{equation}\label{eq:04}
	\begin{tikzcd}[cells={nodes={minimum height=2em}}]
	Q_n \arrow[r]\arrow[d,"\tau_n"] & \cdots \arrow[r] & Q_1 \arrow[r] \arrow[d, "\tau_1"] &  Q_0   \arrow[d, "\tau_0"]   \\
	RG \otimes_{RH} Q_n \arrow[r]\arrow[d, "\jmath_n"] & \cdots \arrow[r] & RG\otimes_{RH}Q_1 \arrow[r] \arrow[d, "\jmath_1"] & RG \otimes_{RH} Q_0   \arrow[d, "\jmath_0" ]    \\
	M_n \arrow[r] & \cdots \arrow[r] & M_1 \arrow[r] & M_0  .
	\end{tikzcd}
	\end{equation} 
	where $\tau_k \colon Q_k \to RG \otimes_{RH} Q_k$ is the natural inclusion given by $q \mapsto e \otimes q$  (here $e$ denotes the identity element of $G$), and the vertical arrows $\jmath_i\colon RG\otimes_{RH} Q_i \to M_i$ are the natural inclusions. Then define \[ \imath_i=\jmath_i\circ \tau_i \]
for $0\leq i\leq n$, and observe that they are  injective $RH$-morphisms. 
	
	\begin{step}
		Verifying commutative diagram~\eqref{eq:03}.
	\end{step}
	
	In view of the commutative diagram~\eqref{eq:04}, we only need to verify that if $H_0(Q)$ and $H_0(M)$ denote the cokernels of  $Q_1\to Q_0$  and  $M_1\to M_0$ respectively, then the $RH$-morphism $H(\imath_0)\colon H_0(Q) \to H_0(M)$ induced by $\imath_0$ is an isomorphism. 
	
Before the argument, we remark that this is not immediate, it depends on the choice of the $RG$-morphism $f_0$; the available choices for $f_0$ depend on the choice of the $RG$-morphism $\phi\colon RG\otimes_{RH} R \to R$; our choice is  defined by~\eqref{eq12}. 

Let $H_0(P)$ denote the cokernel of $P_1\to P_0$. Let $\tau_{-1}\colon R \to RG\otimes_{RH} R$ be defined by $r\mapsto e\otimes r$ where $e$ denotes the identity element of $G$. Then $\phi \circ \tau_{-1}$ is the identity map on $R$. It follows that the induced $RH$-morphism $H_0(f_0\circ\tau_0) \colon H_0(Q)\to H_0(P)$ is an isomorphism. 	Since $\kappa \colon M_* \to P_*$ given by $(p,q,q') \mapsto p + f(q)$ is a chain homotopy equivalence, $H(\kappa_0) \colon H_0(M) \to H_0(P)$ is an isomorphism.  Observe that $H(f_0 \circ \tau_0)$ equals $H(\kappa_0) \circ H(\imath_0)$ and hence $H(\imath_0)$ is an isomorphism.
	
	\begin{step}
		The exact sequence~\eqref{eq:07} splits, and each $S_i$ is a projective $RH$-module. 
	\end{step}
This is immediate since $\imath_i \colon Q_i \to M_i$ is the inclusion of a direct summand of $M_i$ as an $RH$-module. Since restriction of scalars preserves projectiveness, $M_i$ is projective as an $RH$-module and hence $S_i$ is projective as well. 
\end{proof}

\setcounter{step}{0}
\begin{proof}[Proof of Theorem~\ref{thm:main2}]
 Consider projective resolutions as~\eqref{eq:02} and~\eqref{eq:01} as well as $RH$-morphisms $\imath_i\colon Q_i\to M_i$ as described in Proposition~\ref{prop:MappingCylinder}.  
	
	Let $M_*=(M_i, \delta_i^M)$ denote the chain complex induced by~\eqref{eq:01}, with the assumption that $M_i=0$ for $i>n$ and $i<0$. Analogously, let $Q_*=(Q_i, \delta_i^Q)$ be the chain complex induced by~\eqref{eq:02}, with the assumption that $Q_i=0$ for $i>n$ and $i<0$. Observe that we are not using the modules $Q_{n+1}$ and $M_{n+1}$ in the definition of $Q_*$ and $M_*$. Let $S_*$  be the quotient chain complex $M_*/Q_*$. Consider the induced chain map $\imath=(\imath_i)\colon Q_*\to M_*$. 
	
	We use the following notation. The kernel of $\delta_n^Q$ is denoted by $Z_n(Q)$. The $n$-homology group of the complex $Q_*$ is denoted by $H_n(Q)$. Analogous notation is used for the other chain complexes.
	
	\begin{step}\label{step:newretraction}
		The induced sequence 
		\begin{equation} \label{eq:20}
		\begin{tikzcd}[cells={nodes={minimum height=2em}}]
		0 \arrow[r] & Z_n(Q) \arrow[r, "\imath_n"] & Z_n(M) \arrow[r] & Z_n(S) \arrow[r] & 0
		\end{tikzcd}
		\end{equation}
is exact and satisfies:		
\begin{itemize}
		\item  $Z_n(Q)$ is a finitely generated $RH$-module. 
		\item  $Z_n(M)$ is a finitely generated and projective $RG$-module.
		\item  $Z_n(Q)$ is a direct summand of $Z_n(M)$ as an $RH$-module.
	\end{itemize}
 
	\end{step}
	
	Observe that $H_{n+1}(Q)$ and $H_{n-1}(Q)$ are both trivial. The short exact sequence of chain complexes of $RH$-modules
	\begin{equation}\label{eq:09}
	\begin{tikzcd}[cells={nodes={minimum height=2em}}]
	0 \arrow[r] & Q_* \arrow[r, "\imath"] & M_* \arrow[r] & S_* \arrow[r] & 0
	\end{tikzcd}
	\end{equation}
	induces a long exact sequence
	\begin{equation} 
	\begin{tikzcd}[cells={nodes={minimum height=2em}}]
	0 \arrow[r] & H_n(Q) \arrow[r, "\imath_n"] & H_n(M) \arrow[r] & H_n(S) \arrow[r] & 0
	\end{tikzcd}
	\end{equation}
	which is precisely~\eqref{eq:20}.

	The $RH$-module $Z_n(Q)$ is  finitely generated since $Q_{n+1}$ is a finitely generated $RH$-module and $\delta_{n+1}^Q$ maps $Q_{n+1}$ onto $Z_n(Q)$. 
	
	That $Z_n(M)$ is a finitely generated and  projective $RG$-module follows from a direct application of Schanuel's lemma~\cite[pg.193, Lemma 4.4]{Br94} to the exact  sequences~\eqref{eq:01} and 
	\begin{equation}\label{eqa}
	0 \rightarrow Z_n(M )\rightarrow M_n \rightarrow \cdots \rightarrow M_0 \rightarrow R\rightarrow 0.
	\end{equation}
	Finally, to show  that $Z_n(Q)$ is a direct summand of $Z_n(M)$ as an $RH$-module, we argue that 
that $Z_n(S)$ is projective $RH$-module. Consider the sequence of $RH$-modules induced by $S_*$
	\begin{equation}\label{eq8}
	0 \rightarrow Z_n(S) \rightarrow S_n \rightarrow \cdots \rightarrow S_0 \rightarrow 0.
	\end{equation}
	Note that this sequence is exact by observing the long exact sequence of homologies induced by~\eqref{eq:09}. Indeed, $H_i(Q)$ and $H_i(M)$ are trivial for $0<i<n$, and $H(\imath)\colon H_0(Q) \to H_0(M)$ is an isomorphism by~\eqref{eq:03}. Since each $S_i$ is projective, exactness of~\eqref{eq8} implies that $Z_n(S)$ is projective. 
	
	\begin{step}
		$FV_{H, R}^{n+1} \preceq FV_{G, R}^{n+1}$.
	\end{step}
	Let $\|\cdot\|_{M_n}$ and $\|\cdot \|_{Z_n(M)}$ denote filling norms on the $RG$-modules $M_n$ and $Z_n(M)$ respectively. Similarly, let $\|\cdot\|_{Q_n}$ and $\|\cdot \|_{Z_n(Q)}$ denote filling norms on $RH$-modules $Q_n$ and $Z_n(Q)$. 
	For the map $Z_n(Q)  \xrightarrow{\imath_n} Z_n(M)$, by Lemma \ref{lem: integral part H} there exist integral parts $K$ and $K'$ of $Z_n(Q)$ and $Z_n(M)$ respectively, such that $K$ maps into $K'$ by the morphism $\imath$. \par

Since $\imath\colon Q_n \to M_n$ is the inclusion of a direct summand of $M_n$ as an $RH$-module, and $M_n$ is a projective $RH$-module, 
Lemma~\ref{prop:RetractionLemma} implies that $\| \cdot \|_{M_n} \sim   \| \cdot\|_{Q_n}$ on $Q_n$. In particular,  there is a constant $C_0$ such that 
	\[ \| \imath_n (\gamma)\|_{M_n} \leq C_0 \| \gamma \|_{Q_n} \]
for every $\gamma \in Q_n$.  

By  Step~\ref{step:newretraction},  $\imath_n\colon Z_n(Q)  \to Z_n(M)$ is the inclusion of a direct summand of $Z_n(M)$ as an $RH$-module, and $Z_n(M)$ is a projective $RH$-module. Lemma~\ref{prop:RetractionLemma} implies $\| \cdot \|_{Z_n(M)} \sim   \| \cdot\|_{Z_n(Q)}$ on $Z_n(Q)$. Hence there is $C_1>0$  such that 
	\begin{equation} \nonumber
	\|\gamma \|_{Z_n(Q)} \leq C_1\|\imath_n (\gamma)\|_{Z_n(M)}
	\end{equation}
for every $\gamma \in Z_n(M)$, and $\rho \circ \imath$ is identity on $Z_n(Q)$.

Let $k \in \mathbb{N}$ and $\gamma \in K \subseteq Z(Q_n)$ such that $\|\gamma\|_{Q_n} \leq k$.
	Then
\[  \|\gamma\|_{Z_n(Q)}  \leq C_1 \|\imath_n(\gamma)\|_{Z_n(M)}  \leq  C_1\FV_{G,R}^{n+1}( \|\imath_n(\gamma)\|_{M_n}) \leq C_1\FV_{G,R}^{n+1}( C_0\|\gamma\|_{Q_n}) \]
Therefore 
	$\FV_{H, R}^{n+1}(k)\leq C_1\FV_{G,R}^{n+1}(C_0k)$ for every $k$.
\end{proof}

\bibliographystyle{plain}
\bibliography{ref}

\end{document}